\documentclass[12pt,a4paper]{amsart}
\usepackage[latin1]{inputenc}
\usepackage[english]{babel}
\usepackage{amsmath,amssymb,amsthm}
\usepackage{tikz}
\usetikzlibrary{matrix}

\usepackage{cite}

\theoremstyle{plain}
\newtheorem{theorem}{Theorem}[section]
\newtheorem{lemma}[theorem]{Lemma}
\newtheorem{corollary}[theorem]{Corollary}
\newtheorem{prop}[theorem]{Proposition}

\theoremstyle{definition}

\newtheorem{remark}[theorem]{Remark}

\begin{document}

\title[Arithmetic properties of polynomial products]
{\bf Polynomial products modulo primes and applications}
\author{Oleksiy Klurman} 
\address{KTH, Royal Institute of Technology, Stockholm, Sweden }
\email{klurman@dms.umontreal.ca}
\author{Marc Munsch}
\address{5010 Institut f\"{u}r Analysis und Zahlentheorie
8010 Graz, Steyrergasse 30, Graz}
\email{munsch@math.tugraz.at}
\date{\today}


 \begin{abstract} 
 For any polynomial $P(x)\in\mathbb{Z}[x],$ we study arithmetic dynamical systems generated by $\displaystyle{F_P(n)=\prod_{k\le n}}P(n)(\text{mod}\ p),$ $n\ge 1.$ We apply this to improve the lower bound on the number of distinct quadratic fields of the form $\mathbb{Q}(\sqrt{F_P(n)})$ in short intervals $M\le n\le M+H$ previously due to Cilleruelo, Luca, Quir\'{o}s and Shparlinski. As a second application, we estimate the average number of missing values of $F_P(n)(\text{mod}\ p)$
 for special families of polynomials, generalizing previous work of Banks, Garaev, Luca, Schinzel, Shparlinski and others.
 \end{abstract}

\maketitle

\section{Introduction}
Let $P(x)\in\mathbb{Z}[x]$ be a non constant polynomial and define $$F_{P}(n)=\prod_{i=1}^nP(i).$$  
Throughout the years, arithmetic properties of the function $F_P(n)$ attracted considerable attention of several authors. Perhaps the first occurrence of this object in the literature dates back to Chebyshev, who considered the case $P_0(n)=n^2+1$ and showed that the largest prime factor of $F_{P_0}(n)$ is $\gg n.$ See also [12], [13], [15] for refinements of the latter result.
Another direction in this investigation stems from equations of the form 
\begin{align}\label{key1}
F_{P}(n)=m^k
\end{align}
where $m,n\in\mathbb{N},$ $k\ge 2.$ Erd\"os and Selfridge~\cite{Erdos} showed that~\eqref{key1} has finitely many solutions for $P(x)=x+a.$ The same question for $P(x)=ax+b$ was the subject of further investigation ( \cite{Gyory}, \cite{Gyoryperfect}, \cite{Gyoryprog}, \cite{Laishram}, \cite{Shorey}). Using only elementary arguments, Cilleruelo~\cite{Cillsquares}  was the first to handle the case of polynomials of degree $2$ (precisely $P(x)=x^2+1$).\\ Several authors adapted his method to deal with more general quadratic polynomials or reducible higher degree polynomials (\cite{Chinesexl1},   \cite{4k+1}, \cite{Turkishguys}, \cite{Thaisquare}, \cite{Zhangpowerful}, \cite{Zhangsquares},\cite{Zhangpow}). However, no specific result for irreducible polynomials of degree $d\geq 3$ is known. Motivated by the previous work~\cite{CillShpar}, we study the number of solutions of \eqref{key1} on average over the short intervals.
\\
\subsection{Results on average.}\label{averageresults} For $M,N\ge 1$ and $d\geq 1$ we denote by $S_d(M,N)$ the set of integers $n$ such that there exists $t\geq 1$ satisfying
 
\begin{equation}\label{squares} F_{P}(n)=dt^2,  \text{      for    }  n=M+1,\dots,M+N. \end{equation} 
Cilleruelo, Quir\'{o}s and Shparlinski investigated the averaged version of the problem~\eqref{squares}. In particular, they proved in \cite{CillShpar} that uniformly
$$\# S_d(M,N) \ll N^{11/12} (\log N)^{\frac{1}{3}}.$$
Our first goal in this note is to improve this result.
\begin{theorem}\label{averagesquares}  Let $P(x) \in \mathbb{Z}[x]$ be an irreducible polynomial with $\text{deg } P \ge 2.$ Then, uniformly for squarefree integers $d\geq 1$ and arbitrary integers $M,N \geq 1$, we have 
$$\# S_d(M,N) \ll N^{7/8} (\log N)^{1/4},$$
where the implied constant depends only on the degree of $P.$
\end{theorem}
As a corollary we obtain 
\begin{corollary} Let $P(x) \in \mathbb{Z}[x]$ be an irreducible polynomial with $\text{deg } P \ge 2.$ Then, there are at least $\gg \frac{N^{7/8}}{(\log N)^{1/4}}$ distinct quadratic fields amongst $\mathbb{Q}(\sqrt{F_P(n)})$ for $n=M+1,\dots,M+N$.    \end{corollary}This improves on the previous bound given in~\cite{CillShpar}.
Our second objective is to study the distribution of the values of polynomial products modulo primes.

\subsection{Distribution of polynomials products modulo $p$}\label{distrib}
The case $P(x)=x$ for which $F_P(n)=n!$ has been extensively investigated in~\cite{Turkish}, \cite{Broughan}, \cite{Zaha}, \cite{GaraMona}, \cite{KLM1} and \cite{Lev}. Our main motivation is the following question: \\

\noindent {\bf \underline{Question}:} {\it Given a prime $p$, what can be said about the cardinality  $$G_{P}(p)=\big\lvert\left\{F_{P}(n) \,(\bmod\,p),n=1\dots p\right\}\big\rvert?$$}
This is a part of a general program concerning images of dynamical system of non-algrebraic origin. See the survey paper~\cite{survey} for the extensive list of the related problems and references therein. 
Although $G_P(p)$ is a natural generalization of the analogous quantity for factorials (namely, for $P(x)=x$), in general the asymptotic behaviour might be rather different.
If $P(n)$ has a root modulo $p$ and $n_0$ denotes the smallest of such roots, we have  $$F_P(n)\equiv 0 \,(\bmod\,p) ,$$
 for all $n\ge n_0.$ In this case, $G_{P}(p)\le n_0.$ Furthermore, we know since Nagel \cite{Nagelfr} that there exists very large prime factors dividing polynomial products. This culminates with the work of Tenenbaum \cite{TenCheb} who showed
$$P^{+}(F_{P}(n)) \gg n\exp((\log n)^{\alpha} \mbox{ for each }0<\alpha<2-\log 4$$ where the implied constants depend only on $\alpha$ and $P$. Choosing such large primes $p$, we derive $$ G_{P}(p)\ll p \exp(-(\log p)^{\alpha})$$ for every $\alpha$ as above. In order to circumvent this problem, we confine ourselves with the case where $p$ does not divide $P(n),$ for all $n\in\mathbb{Z}.$ This brings significant complications to the analysis below. One might speculate that
\[G_P(p) \sim C_pp.\]
This is in line with the conjecture of Erd\H{o}s and Schinzel~\cite{RokoSchinzel} for the case $P(x)=x$ and $C_p=\left(1-\frac{1}{e}\right).$ The random model of this problem, namely when $n!\pmod p$ is replaced by the product $\pi(1)\pi(2)\dots \pi(n)\pmod p$ where $\pi$ is a random permutation has been studied in the works~\cite{Zaha} and~\cite{Lev}. 
 
In the last section of the paper, we show that $G_P(p)$ is not too large on average over these ``good" primes, in other words, $F_P(n)$ misses a lot of values modulo $p.$ This generalizes previous works for the case of factorials~\cite{Turkish}, \cite{KLM1}. Similar results have been obtained for some others maps $n\to F(n)\pmod p,$ see the survey~\cite{survey} for more details.

For the sake of proving such results, we restrict ourselves to several families of polynomials. First, suppose that $P\in \mathbb{Z}[x]$ is such that its splitting field (which we denote by $Spl(P)$) is an imaginary quadratic extension of $\mathbb{Q}$. Next theorem shows that the number of ``missing" residue classes tends to infinity on average for such polynomials.

\begin{theorem}\label{average}
For any imaginary quadratic $P(x)\in\mathbb{Z}[x],$ we have 
 
 $$\frac{1}{\pi(x)}\sum_{p\leq x\atop \text{P has no root mod p}} \left(p-G_{P}(p)\right) \gg \frac{\log\log \log x}{\log\log\log\log x}\cdot$$
 \end{theorem}

Theorem~\ref{average} directly implies:
\begin{corollary}\label{inf}
 There exists infinitely many primes $p$ such that $P$ does not have a root modulo $p$ and 
\[p-G_{P}(p)\gg \frac{\log \log \log p}{\log \log \log \log p}. \] 
\end{corollary}

Assuming Generalized Riemann Hypothesis (GRH) the bound from Theorem ~\ref{average} can be improved.  
  \begin{theorem}\label{GRH} 
Assume that GRH is true. Then, 
 
 $$\frac{1}{\pi(x)}\sum_{p\leq x \atop \text{P has no root mod p}} \left(p-G_{P}(p)\right) \gg \log x.$$
 \end{theorem}

As before, Theorem~\ref{GRH} directly implies:
\begin{corollary}\label{GRHinf}
Assume that GRH is true. There exists infinitely many primes $p$ such that $P$ does not have a root modulo $p$ and 
\[p-G_{P}(p)\gg \log p. \] 
\end{corollary}
We prove analogue results for binomials of any degree which, in some sense, generalizes the case of imaginary quadratic polynomials. Even though the method is essentially similar, the algebraic properties of the polynomials involved and their utilization in the proof are different.
 \begin{theorem}\label{averageBinomial}
Let $P(x)=x^d-a\in\mathbb{Z}[x]$ with $d$ coprime to $a$ and $a\neq \pm 1$ squarefree. Then
 
 $$\frac{1}{\pi(x)}\sum_{p\leq x\atop \text{P has no root mod p}} \left(p-G_{P}(p)\right) \gg_{d,a} \frac{\log\log \log x}{\log\log\log\log x}.$$
 \end{theorem}
 
Again, Theorem~\ref{averageBinomial} directly implies:
\begin{corollary}\label{infBinomial}
 There exists infinitely many primes $p$ such that $P$ does not have a root modulo $p$ and 
\[p-G_{P}(p)\gg_{d,a} \frac{\log \log \log p}{\log \log \log \log p}. \] 
\end{corollary}


In Theorem \ref{averageBinomial} and Corollary \ref{infBinomial}, the dependence on $a$ is in fact a dependence on the smallest prime factor of $a$. Furthermore, the condition $d$ coprime to $a$ can be weakened. Finally, akin results under the assumption of the Generalized Riemann Hypothesis can also be deduced.

The paper is organized as follows. In Section \ref{averageproof}, we prove the results of Section \ref{averageresults}. We prove the necessary algebraic results concerning a suitable family of shifted polynomials in Section \ref{distribproof} and conclude with the proofs of results announced in Section \ref{distrib}.

\section{Number of solutions of (\ref{key1}) in intervals}\label{averageproof}

We proceed using the square sieve in a different way than in \cite{CillShpar}. In order to prove our theorem, we now collect a few auxiliary lemmas from~\cite{CillShpar} and incorporate an additional combinatorial input. \\

 \subsection{Technical lemmas}

The following lemma can be deduced from the Weil bound and is stated in that way in \cite[Lemma $4$]{CillShpar}.

\begin{lemma}\label{Weil} Let $P(x) \in \mathbb{Z}[x]$ be an arbitrary polynomial with $D=\text{deg }P\ge 2.$ For all primes $l\neq p$ such that $P(x)$ is not a perfect square modulo $l$ and $p$, we have

$$ \sum_{n=M+1}^{M+N} \left(\frac{P(n)}{lp}\right) \ll D^2\left(\frac{N}{lp}+1\right) (lp)^{1/2} \log (lp).                 $$

\end{lemma}

\begin{remark} The result remains true for linear polynomials and is, in that case, a direct consequence of the P\'{o}lya-Vinogradov inequality. \end{remark}

We recall a consequence of Chebotarev Density Theorem. Denote by $\mathcal{L}_z$ the set of primes in the interval $ [z,2z]$ such that $P(x)$ has no root modulo $l$.

\begin{lemma}\label{Chebo} Let $P(x) \in \mathbb{Z}[x]$ be an irreducible polynomial of degree $D\geq 2$. We have 

$$\# \mathcal{L}_z  = \frac{1}{\kappa}(\pi(2z)-\pi(z)) + O\left(z(\log z)^{-2}\right),$$ where $\kappa \leq \frac{D!}{(D-1)}$ is a positive rational number depending on the polynomial $P(x)$ and  $\pi(x),$ as usual denotes the number of primes $\le x.$ \end{lemma}

For an integer $h \geq 1$, we introduce

$$F_{h}(x)=\prod_{n=1}^{h} P(x+n) \in \mathbb{Z}[x].$$In order to apply Weil's inequality, we need to ensure that these particular polynomials are not squares modulo a lot of primes. The next result follows directly from the proof of the main Theorem of \cite{CillShpar}.

\begin{lemma}\label{goodprimes} Let $P(x) \in \mathbb{Z}[x]$ be an irreducible polynomial of degree $D\geq 2$ and $H \geq 1$. There is at most $O\left(H\log H/\log\log H\right)$ primes $p$ such that $F_h$ is a square modulo $p$ for some $0<h\leq H$. \end{lemma}

\subsection{Proof of Theorem \ref{averagesquares}}

Fix $H$ as a parameter which will be chosen later. Write $S_d(M,N) = S_1 \cup S_2$ where  $$S_1=\{n\in S_d(M,N), m-n>H \text{ for all  } m>n \in S_d(M,N)\}$$ and let $S_2=S_d(M,N)\backslash S_1$. Clearly we have $\# S_1 \leq N/H$ and we want to upper bound the cardinality of $S_2$ in terms of $H$. First, remark that if $n\in S_2$, this implies that there exists $h\leq H$ such that $n$ and $n+h$ are solutions of the equation $(\ref{squares})$ and it follows easily that $F_h(n)$ is a square. Therefore $\left(\frac{F_h(n)}{l}\right)=1$ for all primes not dividing $F_h(n)$. Particularly, this holds for primes $l$ such that $P$ has no roots modulo $l$. For a parameter $z$ which will be chosen later as well, we can choose such primes using Lemma \ref{Chebo}. Let us denote by $\mathcal{L}_z$ this set of good primes. Using Lemma \ref{goodprimes} we have that for sufficiently large $z$, there is at least half of the primes $l\in L_z$ such that $F_{h}(x) $ is not a perfect square modulo $l$ for every $h\leq H$ ($z$ will be chosen such that $H\log H = o(z/\log z)$). We denote this set of primes by $\mathcal{P}_z$. Thus, we have that for every $n\in S_2$, there exists $h\leq H$ such that $$ \sum_{l\in \mathcal{P}_z} \left(\frac{F_h(n)}{l}\right)=  \# \mathcal{P}_z.$$Hence, averaging over the interval we get

$$ (\# \mathcal{P}_z)^2 \# S_2 \ll \sum_{n=M+1}^{M+N}\sum_{1\leq h\leq H} \left( \sum_{l\in \mathcal{P}_z}\left(\frac{F_h(n)}{l}\right)\right)^2.$$ Separating the diagonal contribution from the nondiagonal terms, we obtain 

$$ (\# \mathcal{P}_z)^2 \# S_2 \ll N H\# \mathcal{P}_z +  \sum_{l,p \in \mathcal{P}_z \atop l\neq p}\sum_{1\leq h\leq H}\sum_{n=M+1}^{M+N}\left(\frac{F_h(n)}{lp}\right). $$ Applying Lemma \ref{Weil}, we derive

$$ \# S_2  \ll \frac{NH}{\# \mathcal{P}_z}  + (\# \mathcal{P}_z)^{-2}\sum_{l,p \in \mathcal{P}_z \atop l\neq p} \sum_{1\leq h\leq H} h^2(lp)^{1/2}\left(\frac{N}{lp}+1\right) \log (lp).$$ It leads to

$$ \# S_2 \ll \frac{NH}{\# \mathcal{P}_z} + H^3 z \left(\frac{N}{z^2}+1\right) \log z.$$ Setting $z=\sqrt{N}$, we end up with $$\# S_2 \ll  \frac{NH(\log N) }{\sqrt{N}} + H^3 \sqrt{N} \log N.$$ Choosing the optimal parameter $H= \frac{N^{1/8}}{(\log N)^{1/4}}$ concludes the proof.

\section{Value distribution of polynomials products modulo $p$}\label{distribproof}
 As was mentioned in the introduction, we are interested in the quantity $G_{P}(p).$ We have a trivial lower bound.\footnote{Improving this bound seems hard to the authors.}
\begin{prop}
Suppose $p$ is a prime that does not divide $P(n)$ for all $n\in\mathbb{Z}.$ Then  $$G_P(p)\ge\sqrt{\frac{p}{\deg P}}\cdot$$ 
\end{prop}
\begin{proof}
Since every polynomial $P$ has at most $\deg P$ roots, for $p\ge \deg P$ we have that
$$\big\lvert P(n)\,(\bmod\,p)\big\rvert\ge \frac{p}{\deg P}\cdot$$
Now take different values $P(n_1), P(n_2),\dots ,P(n_k)$ such that $P(n_i)\ne P(n_j)\,(\bmod\,p)$ for $i\ne j$. Consider the pairs of the form $(n_i-1,n_i).$ Clearly, 
\[\frac{F_{P}(n_i)}{F_{P}(n_i-1)}=P(n_i)\]
and the result follows in exactly the same way as in the case of factorials \cite{KLM1} which corresponds to $P(n)=n$.
\end{proof}
 
Our goal in this section is to show that $G_P(p)$ is not too large on average for ``good" primes. We restrict ourselves to some specific families of polynomials. In the first part, we assume that the splitting field of $P$ is an imaginary quadratic extension of $\mathbb{Q}$. In the second part, we extend our results to some particular polynomials of any degree.

 \subsection{Algebraic properties of polynomial products}

 We fix a few standard notations from algebraic number theory. Let $L/\mathbb{Q}$ be a Galois extension of number fields and denote by $n_L$ its degree. If $G$ is the Galois group of this extension, we denote by $C$ a subset of $G$ stable by conjugation. For all $x>1$, the function $\pi_C(x)$ will count the number of prime numbers $p \leq x$, non ramified in $L$, and such that $\left[\frac{L/\mathbb{Q}}{p}\right] \in C $, where $\left[\frac{L/\mathbb{Q}}{p}\right]$ is the unique Galois automorphism (up to conjugacy) such that its reduction modulo $p$ coincides with the Frobenius automorphism. In case $C=G$ we have $$\pi_C(x) = \pi(x) \sim \frac{x}{\log x}$$ when $x\rightarrow +\infty$. 
  
  For any ideal $\mathfrak{I}$ of the ring of integers $\mathcal{O}_L$, we denote the norm of an ideal by $N_{L/\mathbb{Q}}(\mathfrak{I})$ and write $N(\mathfrak{I})$. We also denote by $f\left(\mathfrak{p}/p\right)$ the inertial degree $\vert\left[\mathcal{O}_L/\mathfrak{p}:\mathbb{F}_{p}\right]\vert$ of the ideal $\mathfrak{p}$ above the rational prime $p$. The function $\pi_L(x)$ will count the number of prime ideals $\mathfrak{p}$ of $L$ such that $N(\mathfrak{p}) \leq x$. Finally, denote  by $d_L$ the absolute discriminant of the extension $L$. \\
  
  In the rest of the paper, we consider the family of polynomials defined by $f_n(x)=P(x)P(x+1)\dots P(x+n-1)-1, n\geq 1$ and $P(x)\in\mathbb{Z}[x].$

\subsubsection{Imaginary quadratic case}

 In this section, take $P\in \mathbb{Z}[x]$ such that the splitting field of the polynomial $P$ (which we denote by $Spl(P)$) is an imaginary quadratic extension of $\mathbb{Q}$. Denote by $\alpha_1$ and $\alpha_2$ the complex roots of $P$. Hence $Spl(P)=\mathbb{Q}(\alpha_1)$ and denote by $\mathcal{O}_{\mathbb{Q}(\alpha_1)}$ the ring of integers of $\mathbb{Q}(\alpha_1)$.

\begin{prop}\label{unitsimag}
There exists an effective constant $C_{P}$ depending only on $P$ such that $f_n(x)$ is irreducible over $\mathcal{O}_{\mathbb{Q}(\alpha_1)}[x]$ for $n>C_{P}$. 
\end{prop}

\begin{proof}
Observe that $$f_n(\alpha)=-1, \hspace{1cm} \alpha=\alpha_i - j, \hspace{3mm}\,\, i=1,2,\hspace{3mm} 0\leq j\leq n-1.$$

Now, suppose that $f_n(x)=g(x)h(x)$ in $\mathcal{O}_{\mathbb{Q}(\alpha_1)}[x]$ with $\deg g\geq n$ and $\deg h\geq 1$. First, note that the above $2n$ values of $\alpha$ are distinct. Indeed, otherwise, we would have $\alpha_1-j_1=\alpha_2-j_2$ with $j_1\neq j_2$. It would imply $\alpha_1=\alpha_2+k$ with $k$ a nonzero integer which forces $\alpha_1$ to be rational by Vieta's formulas. Hence $g$ has to take units values for $2n$ distinct algebraic integers of $\mathcal{O}_{\mathbb{Q}(\alpha_1)}$. The hypothesis implies that there are at most six units in $\mathcal{O}_{\mathbb{Q}(\alpha_1)}$. Thus, by the pigeonhole principle, we derive that $g(\beta_i)=u$ for at least $n/3$ distinct algebraic integers $\beta_i$\footnote{These are a subset of the previous $\alpha_i-j$.}. Consequently,

\[ g(x)=(x-\beta_1)(x-\beta_2)\cdots(x-\beta_{\frac{n}{3}}) q(x) + u.\]
Suppose that $g$ takes another unit value $u'$ at an algebraic integer point $\beta$. Then

\begin{equation}\label{divisors} (\beta-\beta_1)(\beta-\beta_2)\cdots(\beta-\beta_{\frac{n}{3}})q(\beta)=u'-u. \end{equation}

Define $M=\max\left\{d(y), y=v_1-v_2, v_1,v_2 \text{ units}\right\}$ where $d(y)$ denotes the divisor function in $\mathcal{O}_{\mathbb{Q}(\alpha_1)}.$ If $n>3M$, then immediately get a contradiction.

In fact, in view of (\ref{divisors}), the values $\beta-\beta_j, 1\leq j \leq \frac{n}{3}$ have to be divisors of $u'-u$ and we have at most $M$ of those. We deduce that $g(\alpha)=u$ for all $\alpha_i-j,\,i=1,2,\hspace{3mm} 0\leq j\leq n-1$ and so $h(\alpha)=-1/u$ at these $2n$ values. This is indeed impossible since $\deg h\leq n$.

\end{proof}
We note that the previous argument is standard while working over $\mathbb{Q},$ see for example,~\cite[Theorem $8$]{Dorwart} or \cite[Theorem $5.1$]{PolyaSchur}. We deduce
\begin{corollary}\label{corimag} Suppose that $Spl(P)$ is an imaginary quadratic extension of $\mathbb{Q}$. Let$n>C_{P}$ be as in Proposition \ref{unitsimag} and set $\beta$ to be a root of $f_n$. Then

$$\mathbb{Q}(\beta) \cap Spl(P) = \mathbb{Q}.$$

\end{corollary}

\begin{proof}
Let $\mathbb{Q}(\alpha_1)$ denote the splitting field of $P$.
Suppose that $\mathbb{Q}(\beta) \cap \mathbb{Q}(\alpha_1) \neq \mathbb{Q},$ in which case $\mathbb{Q}(\alpha_1)\subset \mathbb{Q}(\beta).$ By Proposition \ref{unitsimag} $f_n$ is irreducible over $\mathcal{O}_{\mathbb{Q}(\alpha_1)}[x].$ Therefore, $\beta$ is of degree $2n$ over $\mathbb{Q}(\alpha_1)$ and also over $\mathbb{Q}$. 
The following diagram immediately yields a contradiction.

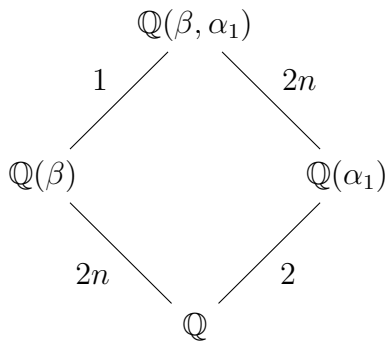
\center{\begin{tikzpicture}[node distance = 2cm, auto]
      \node (Q) {$\mathbb{Q}$};
      \node (E) [above of=Q, left of=Q] {$\mathbb{Q}(\beta)$};
      \node (F) [above of=Q, right of=Q] {$\mathbb{Q}(\alpha_1)$};
      \node (K) [above of=Q, node distance = 4cm] {$\mathbb{Q}(\beta, \alpha_1)$};
      \draw[-] (Q) to node {$2n$} (E);
      \draw[-] (Q) to node [swap] {$2$} (F);
      \draw[-] (E) to node {$1$} (K);
      \draw[-] (F) to node [swap] {$2n$} (K);
      \end{tikzpicture}}

\end{proof}

The argument above heavily relies on the finiteness of the units of the splitting field of $P$, thus it can only work when the roots of $P$ lie in an imaginary quadratic extension. However, we can prove weaker (but sufficient for our applications) algebraic results concerning some  subfamily of polynomials $f_n(x)$ for some other class of polynomials as we will show in the next section.


\subsubsection{Radical extension case}\label{Kummer}
In this section, we consider the case of binomials $P$ of higher degrees which give rise to simple radical extensions. The first two lemmas are well-known.

\begin{lemma}[Theorem $9.1$,\cite{Lang}]\label{Irreducible} 
Let $a$ an element of a field $K$, and $n \in \mathbb{N}.$ Then the polynomial $x^d-a$ 
 is irreducible over $K$ if and only if $a\notin K^p$ for all primes $p\mid d$ and $a\notin -4K^4$ when $4\mid d$.
\end{lemma}
An easy classical computation gives the discriminant of such polynomials.
\begin{lemma}\label{calculdisc}
$$disc(x^d + ax + b)=(-1)^{\frac{d(d-1)}{2}}((1-d)^{d-1}a^d + d^db^{d-1}).$$ In particular, 
$$\vert disc(x^d-a)\vert=d^na^{d-1}.$$
\end{lemma}


Assume $P(x)=x^d-a$ with $a\neq \pm 1$ a squarefree integer. Then for any integer $k\geq 1$, we have 
\begin{align*} f_{kq}(x) \bmod q = P(x)P(x+1)\dots P(x+kq-1)-1 \bmod q \\ =
 (x^d(x+1)^d\dots (x+q-1)^d)^k - 1 \bmod q \end{align*} for every prime divisor $q$ of $a$. Under these conditions, we deduce

\begin{lemma}\label{discshifts} Suppose that $q\vert a$ and $(dk,q)=1.$ Then
$$disc(f_{kq}(x)) \bmod q = disc(f_{kq}(x)\bmod q) \neq 0.$$
\end{lemma}

\begin{proof} Using the previous reduction, we can easily show that the polynomials $f_{kq}(x)$ and $f'_{kq}(x)$ are coprime in $\mathbb{Z}/ q\mathbb{Z}[x]$. Thus, denoting by $Res(f,g)$ the resultant of two polynomials, we have $$Res(f_{kq}(x) \bmod p, f'_{kq}(x)\bmod p) \neq 0.$$ This implies that $disc(f_{kq}(x) \bmod q) \neq 0$ or equivalently $q\nmid disc(f_{kq}(x))$.

 \end{proof}
%
%
Combining all of the above, we arrive at
\begin{prop}\label{indfieldsbinom}
Assume $P(x)=x^d-a$ with $(d,a)=1$ and $a\neq \pm 1$ squarefree. If $k\geq 1$ is such that $(k,a)=1$ and $q$ is a prime divisor of $a$ then $P(x)$ is irreducible over $K_{kq}:=\mathbb{Q}(\beta_{kq})$ where $\beta_{kq}$ is a root of $f_{kq}$.
\end{prop}
\begin{proof} 
We show that the conditions of Lemma \ref{Irreducible} are fulfilled. We only need to verify that $a \notin K_{kq}^{p}$ for $p$ a prime divisor of $d.$ If $a \in K_{kq}^{p}$, then there exists $\beta \in K_{kq}$ such that $a=\beta^p$, in other words the polynomial $x^p-a$ has a root in $K_{kq}$. Thus, $\mathbb{Q}(\beta)$ would be a subextension of $K_{kq}$. By hypothesis, $x^p-a$ has no rational root and thus it is irreducible over $\mathbb{Q}.$ Hence it is the minimal polynomial of $\beta$ and, by Lemma \ref{discshifts}, $q$ divides the absolute discriminant of the field $\mathbb{Q}(\beta)$. Therefore, $q$ should ramify in $K_{kq}$. We know that $disc(K_{kq})$ is a factor of the discriminant of the minimal polynomial of $\beta_{kq}$ which is by definition a divisor of $f_{kq}$. Thus by transitivity, $q$ has to divide $disc(f_{kq})$ which is in contradiction with Lemma \ref{discshifts}.

 \end{proof}

In order to bound discriminants of number fields and locate Siegel zeroes, we need the the following simple result.

\begin{lemma}\label{distance} Take $P(x)=x^d-a \in \mathbb{Z}[x]$ an irreducible polynomial and $d\geq 2.$ If $\beta_i, \,i=1\dots n$ are the complex roots of $f_n,$ then
$$ \vert \beta_i - \beta_j\vert \ll_d n, \hspace{3mm} i \neq j\cdot $$  
  \end{lemma}

\begin{proof} A direct application of Rouch\'{e}'s Theorem to the polynomials $f_n$ and $g=f_n+1$ in the ball $B(0,3(n+\sqrt{a})):=\{z \in \mathbb{C}, \vert z \vert < 3(n+\sqrt{a})\}$ gives the result after noticing that the roots of the polynomial $g$ are exactly $-k + \zeta \sqrt[d]{a}, k=0\dots n-1$ with $\zeta$ being a $d$-th root of unity. \end{proof}

\subsection{Proofs of the results of Section \ref{distrib}}

\subsubsection{Proof of Theorem \ref{average}}

Let $N$ be a parameter which will be determined later. For $n\ge 1$ we consider the family of polynomials $$f_n(t)=P(t)P(t+1)\dots P(t+n-1)-1.$$ By Proposition \ref{unitsimag}, $f_n(t)$ is irreducible over $\mathbb{Q}$ for all $n\ge 1.$ Let  $\rho_n(p)$ denote the number of roots of $f_n(t)$ modulo $p.$ Taking $p$ such that $P$ has no root modulo $p$, we observe that $f_n(t_0)\equiv 0\,(\bmod\,p)$ implies $$F_P(t_0+n-1)=F_P(t_0-1)\,(\bmod\,p).$$ Therefore, each distinct root of $f_n(t)$ modulo $p$ increases the number of ``missing" values by $1$, so we get 

\begin{align}\label{missing} \frac{1}{\pi(x)}\sum_{p\leq x \atop P \text{ has no root mod }p} \left(p-G_{P}(p)\right)&\geq \sum_{n=1}^{N}\frac{1}{\pi(x)}\sum_{p\leq x \atop P \text{ has no root mod }p}\rho_{n}(p).  \end{align} Our task now is to produce a lot of roots of $f_n$ for many values of $p.$

As before, denote by $\mathbb{Q}(\alpha_1)$ the splitting field of $P$. For $n$ sufficiently large, we know by Corollary \ref{corimag} that

$$\mathbb{Q}(\beta_n) \cap \mathbb{Q}(\alpha_1) = \mathbb{Q} $$ where $\beta_n$ is a root of $f_n$. In the following, we will denote by $K_n$ the extension $\mathbb{Q}(\beta_n)$.

We need a lot of primes $p$ such that $P$ has no root mod $p$ and $f_n$ has a root mod $p$. We will achieve this by producing a lot of primes $\mathfrak{p}$ above $p$ such that $f\left(\mathfrak{p}/p\right)=1$ and $\mathfrak{p}$ does not split in $K_n(\alpha_1)$.
 
 
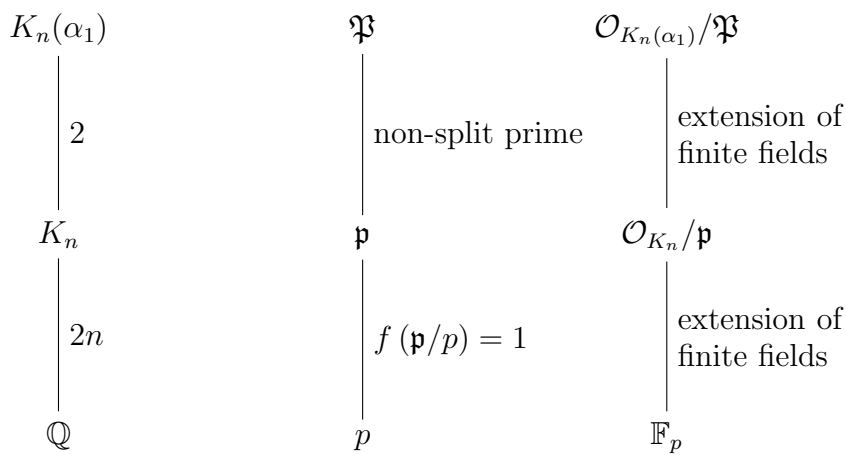
\begin{center}\begin{tikzpicture}
  \matrix (m) [matrix of math nodes,column sep={4cm,between origins},row sep=2cm] {
    K_n(\alpha_1) & \mathfrak{P} & \mathcal{O}_{K_n(\alpha_1)}/\mathfrak{P} \\
   K_n & \mathfrak{p} & \mathcal{O}_{K_n}/\mathfrak{p} \\
     \mathbb{Q} & p &  \mathbb{F}_{p} \\};
  \draw (m-1-1) -- node[right,align=left] {$2$} (m-2-1);
  \draw (m-1-2) -- node[right,align=left] {non-split prime} (m-2-2);
  \draw (m-1-3) -- node[right,align=left] {extension of\\finite fields} (m-2-3);
  \draw (m-2-1) -- node[right,align=left] {$2n$} (m-3-1);
  \draw (m-2-2) -- node[right,align=left] {$f\left(\mathfrak{p}/p\right)=1$} (m-3-2);
  \draw (m-2-3) -- node[right,align=left] {extension of\\finite fields} (m-3-3);
\end{tikzpicture}\end{center}

  Indeed, by Dedekind's theorem we know that, up to finitely many exceptions, $\mathfrak{p}$ splits in $K_n(\alpha_1)$ if and only if $P$ splits in $\mathcal{O}_{K_n}/\mathfrak{p}[x] \cong \mathbb{F}_{p}[x] $ where the last isomorphism comes from the fact that $\mathfrak{p}$ is of inertial degree $1$. Denote by $C$ the set of degree $1$ primes of $K_n$ which does not split in $K_n(\alpha_1)$. By a standard argument, the prime ideals $\mathfrak{p}$ of $K_n$ of inertial degree $f\left(\mathfrak{p}/p\right)>1$ will give negligible contribution. Moreover, we remark that at most $2n$ degree $1$ primes of $K_n$ correspond to the same rational prime. Hence, we get
  
  \begin{equation}\label{ineqroots}\sum_{p\leq x \atop P \text{ has no root mod }p}\rho_n(p) \geq \frac{1}{2n} \sum_{N_{K_n/\mathbb{Q}}(\mathfrak{p})\leq x \atop \mathfrak{p} \text{ does not split in } K_n(\alpha_1)} 1.\end{equation} 
  The strategy is to apply Chebotarev density theorem in the Galois extension $K_n(\alpha_1)/K_n$. We have using \cite{effective} that there exists an absolute constant $c>0$ such that for all $n\ge 1$
\begin{equation}\label{chebouncond}\sum_{N(\mathfrak{p})\leq x,\, f\left(\mathfrak{p}/p\right)=1 \atop \mathfrak{p} \text{ does not split in }K_n(\alpha_1)} 1 = \frac{\pi(x)}{2} + O\left(Li(x^{\beta_n}) + x\exp\left( -c\sqrt{\frac{\log x}{4n}}\right)\right)\end{equation} where $\beta_n$ is the potential positive real zero of the Dedekind zeta function $\zeta_{K_n(\alpha_1)}$ and $$0<1-\beta_n \ll \frac{1}{\log d_{K_n(\alpha_1)}}\cdot$$ 
  
 By a result of Stark~\cite[Theorem $1'$]{Stark}, there exists an absolute constant $c_1$ such that
 
 \begin{equation}\label{Siegel} \beta_n \leq \max\left\{1-\frac{1}{(4n)!\log\vert d_{K_n(\alpha_1)}\vert}, 1-\frac{1}{c_1\vert d_{K_n(\alpha_1)}\vert^{1/4n}}\right\} \end{equation}

 Using (\ref{missing}), (\ref{ineqroots}) together with (\ref{chebouncond}), we derive
\begin{align} \frac{1}{\pi(x)}\sum_{p\leq x \atop P \text{ has no root mod }p} \left(p-G_{P}(p)\right)&\geq  \sum_{n=1}^{N}\frac{1}{4n} + ET\notag \end{align}

where \begin{equation}\label{errorterm} ET= O\left(\frac{1}{\pi(x)}\sum_{n=1}^{N}Li(x^{\beta_n}) + x\exp\left( -c\sqrt{\frac{\log x}{4n}}\right)\right)\cdot \end{equation}

Hence, we have to choose a parameter $N$ such that 

\begin{equation}\label{choiceN}  \log N \gg \sum_{n=1}^{N}\frac{1}{\pi(x)} \left\{Li(x^{\beta_n}) + x\exp\left( -c\sqrt{\frac{\log x}{4n}}\right)\right\} . \end{equation} The sum of the exponential terms in (\ref{choiceN}) satisfies this as long as 

$$N \ll  \frac{\log x}{(\log \log x)^{2}}  . $$ We have the following bound for the discriminant using transitivity formula and Lemma \ref{distance}
  
  \begin{align}\label{discboundtr} d_{K_n(\alpha_1)}&= d_{K_n}^{2}\cdot N_{K_n/\mathbb{Q}}\left(d_{K_n(\alpha_1)/K_n}\right) \notag \\
 & \leq  \prod_{i<j} \vert\beta_i-\beta_j\vert^{4} d^{n} \notag  \\
 & \ll (n^{4})^{4n^2}  d^n\ll n^{64n^2} d^n  \end{align} where $d$ is the discriminant of the quadratic extension $Spl(P)$. To bound the contribution coming from the potential Siegel zeroes, we apply the result of Stark (\ref{Siegel}) together with the discriminant bound (\ref{discboundtr}) to arrive at

 \begin{equation}\label{sumSiegel} \sum_{n=1}^{N}\frac{1}{\pi(x)}Li(x^{\beta_n})\ll \sum_{n=1}^{N} x^{-\frac{1}{n^{c_2n}}} \ll N x^{-\frac{1}{N^{c_2N}}} \end{equation} where $c_2$ is an absolute constant.

 Using the previous bound (\ref{sumSiegel}) and standard computations, we deduce that inequality (\ref{choiceN}) is true as long as 
 
 \begin{equation}\label{finalchoice} N \ll \frac{\log\log x}{\log \log \log x} .   \end{equation} We used twice Dedekind's theorem and so we have to bound the contribution of ``bad" rational primes and of ``bad" rational primes which have a ``bad" prime of degree $1$ above. Thus, the ``bad" rational primes are exactly the primes dividing $\left[\mathcal{O}_{K_{n}}:\mathbb{Z}\left[\beta_n\right]\right]$  or $\left[\mathcal{O}_{K_{n}(\alpha_1)}:\mathcal{O}_{K_{n}}\left[\alpha_1\right]\right]\leq 2$. We have at most $\omega(2n)\ll \log n$ of such primes and, using (\ref{missing}), their total contribution is at most 


\begin{equation*}\sum_{n=1}^{N}\frac{1}{\pi(x)}\sum_{p\leq x \atop p \text{` bad" }}\rho_{n}(p) \ll \frac{1}{\pi(x)}\sum_{n=1}^{N} n\log n \ll \frac{N^2\log N}{\pi(x)} = o(\log N). \end{equation*}

 

\begin{remark} We remark that we are not working directly in the compositum of the splitting fields of $P$ and $f_n$. The reason is that it is much harder, in general, to prove the ``independence" of the splitting conditions.  \end{remark}

\subsubsection{Proof of Theorem \ref{GRH}}

We consider as before the family of polynomials $f_n$ and the associated family of extensions $K_n$ of degree $2n.$ Following the same lines as in the proof of Theorem \ref{average} and replacing the error term in Chebotarev density theorem by the conditional one (see \cite{effective}), we obtain

\begin{equation}\label{chebogrh}\sum_{N(\mathfrak{p})\leq x,\, f\left(\mathfrak{p}/p\right)=1 \atop \mathfrak{p} \text{ does not split in }K_n(\alpha_1)} 1 = \frac{\pi(x)}{2} + O\left(x^{1/2}(\log d_{K_n(\alpha_1)} + 4n\log x)\right).\end{equation} 

Averaging over the family of polynomials $\left\{f_n(x), 1\leq n\leq N\right\}$ and performing the same computation as in the proof of Theorem \ref{average}, we arrive at

\begin{align} \frac{1}{\pi(x)}\sum_{p\leq x \atop P \text{ has no root mod }p} \left(p-G_{P}(p)\right)&\geq  \sum_{n=1}^{N}\frac{1}{4n} + ET \notag \end{align}

where \begin{equation}\label{errortermgrh} ET= O\left(\frac{1}{\pi(x)}\sum_{n=1}^{N}\left\{x^{1/2}(\log d_{K_n(\alpha_1)} + 4n\log x)\right\}\right).  \end{equation}  Hence, we have to choose a parameter $N$ such that
\begin{equation}\label{choiceparameterGRH} \log N \gg \frac{1}{\pi(x)}\sum_{n=1}^{N}\left\{x^{1/2}(\log d_{K_n(\alpha_1)} + 4n\log x)\right\}.\end{equation} Using the discriminant bound (\ref{discboundtr}), we can bound error term by 
  
  $$ET \ll \sum_{n=1}^{N} \left\{x^{-\frac{1}{2}}\log x(\log(n^{64n^2})+ n\log x)\right\} \ll \frac{\log x}{\sqrt{x}} N^3 .$$ An easy computation shows that the error term is negligible compared to $\log N$  provided that $N \ll x^{1/6}$ and the result follows. 

Arguing as in the proof of Theorem \ref{average}, we can easily deal with the contribution of ``bad" primes.

\subsubsection{Proof of Theorem \ref{averageBinomial}}
As in subsection \ref{Kummer}, we assume that $P(x)=x^d-a\in\mathbb{Z}[x]$ with $d$ coprime to $a$ and $a\neq \pm 1$ squarefree.
The proof follows the same lines as the proof of Theorem \ref{average} and thus we merely sketch some of the modifications required. 

Denoting by $K_n$ the extension $\mathbb{Q}(\beta_n)$ with $\beta_n$ a root of the polynomial $f_n$, we want to apply Chebotarev theorem in the Galois extension $(Spl(P),K_n)/K_n$. Without loss of generality, we can assume that $f_n$ is irreducible\footnote{The irreducible case is in fact the worst. Indeed, if $f_n$ has a lot of factors, it will produce many more roots modulo $p$ and consequently many missing values.} over $\mathbb{Q}$. Indeed, we can replace $f_n$ by an irreducible factor $g_n$ and look for primes $p$ such that $g_n$ has a root modulo $p$ (and so $f_n$). 

Thus, we can assume that $K_n$ is a field obtained by adjoining a root of an irreducible polynomial and apply Dedekind's theorem to relate prime ideals of degree $1$ in $K_n$ with roots of $g_n$ modulo $p$. In order to use results from subsection \ref{Kummer}, we restrict ourselves to the subfamily of polynomials $f_{kq}$ with $(k,q)=1$ and $q$ being the smallest prime factor of $a$. 

We know by Proposition \ref{indfieldsbinom} that $P$ is irreducible over $K_{kq}$ and consequently the Galois group of $P$ over $K_{kq}$ is transitive. Thus, it contains a conjugacy class $C$ consisting of elements without fixed point. Hence, by Chebotarev theorem there exists a positive proportion of primes (depending only on $d$ the degree of $P$) of degree $1$ (over $\mathbb{Q}$) primes $\mathfrak{p}$ of $K_{kq}$ such that the associated Frobenius lies in $C$. 

Similarly as in the proof of Theorem \ref{average}, these prime ideals are associated to rational primes $p$ such that $P$ has no root modulo $p$ and $f_{kq}$ has a root modulo $p$. More precisely, at most $n_{K_{kq}} \leq dkq$ of these ideals lies above the same rational prime $p$. 

Averaging over the family of polynomials $\left\{f_{kq}(x), 1\leq k\leq N\right\}$, we arrive at a similar inequality as (\ref{ineqroots}). Performing the same kind of computation as in the proof of Theorem \ref{average} concludes the proof. 

\begin{remark} The main difference comes from a factor term $1/dq$ which arises because we work in a thinner family of polynomials. This gives the dependence in terms of the degree of the original polynomial $P$ as well as its coefficients in the statement of Theorem \ref{averageBinomial}. \end{remark}



\section*{Acknowledgements}

The authors would like to thank Andrew Granville and Igor Shparlinski for valuable remarks.
M.M greatly acknowledges support of the Austrian Science Fund (FWF), START-project Y-901 ``Probabilistic methods in analysis and number theory" headed by Christoph Aistleitner.


\end{document}